\documentclass[11pt, reqno]{amsart}

\usepackage{ amssymb, amsmath, amsthm,  graphicx, psfrag}
\usepackage[usenames,dvipsnames]{color}
\usepackage{enumerate} 
\usepackage{tikz}

\newcommand{\R}{{\mathbb R}}

\def\ab#1{\left|#1\right|}
\def\inpro#1{\left\langle #1 \right\rangle}
\def\Set#1{\left\{\,#1\,\right\}}\def\gp#1{\left(#1\right)}
\def\bk#1{\left[#1\right]}
\def\gp#1{\left(#1\right)}

\newtheorem{theorem}{Theorem}[section]
\newtheorem{proposition}[theorem]{Proposition}
\newtheorem{lemma}[theorem]{Lemma}
\newtheorem{corollary}[theorem]{Corollary}
\newtheorem{conj}[theorem]{Conjecture}

\theoremstyle{definition}
\newtheorem{definition}{Definition}[section]

\theoremstyle{remark}
\newtheorem{remark}{Remark}
\theoremstyle{observation}
\newtheorem{observation}{Observation}

\newcommand{\cupdot}{\charfusion[\mathbin]{\cup}{\cdot}}
\newcommand{\bigcupdot}{\charfusion[\mathop]{\bigcup}{\cdot}}

\makeatletter
\def\moverlay{\mathpalette\mov@rlay}
\def\mov@rlay#1#2{\leavevmode\vtop{%
   \baselineskip\z@skip \lineskiplimit-\maxdimen
   \ialign{\hfil$\m@th#1##$\hfil\cr#2\crcr}}}
\newcommand{\charfusion}[3][\mathord]{
    #1{\ifx#1\mathop\vphantom{#2}\fi
        \mathpalette\mov@rlay{#2\cr#3}
      }
    \ifx#1\mathop\expandafter\displaylimits\fi}
\makeatother

\begin{document}

\title{Minimal scalings and structural properties of scalable frames}

\author{Alice Chan}
\address{Department of Mathematics, University of California San Diego}
\email{azc02010@mymail.pomona.edu}

\author{Rachel Domagalski}
\address{Department of Mathematics, Central Michigan University}
\email{Domag1rj@cmich.edu}

\author{Yeon Hyang Kim}
\address{Department of Mathematics, Central Michigan University}
\email{kim4y@cmich.edu}

\author{Sivaram K. Narayan}
\address{Department of Mathematics, Central Michigan University}
\email{naray1sk@cmich.edu}

\author{Hong Suh}
\address{Department of Mathematics, Pomona College}
\email{Hong.suh@pomona.edu}

\author{Xingyu Zhang}
\address{Department of Industrial Engineering \& Operations Research, Columbia University}
\email{xz2464@columbia.edu}

\thanks{Research supported  by NSF-REU grant DMS 11-56890.}


\date{ 2016.}


\begin{abstract}
For a unit-norm frame $F = \{f_i\}_{i=1}^k$ in $\R^n$, a scaling is a  vector 
$c=(c(1),\dots,c(k))\in \R_{\geq 0}^k$ such that $\{\sqrt{c(i)}f_i\}_{i =1}^k$ is a Parseval frame in $\R^n$. 
If such a scaling exists, $F$ is said to be scalable. 
A scaling $c$ is a minimal scaling if $\{f_i : c(i)>0\}$ has no proper scalable subframe. 
In this paper, 
we provide an algorithm to find all possible  contact points for the John's decomposition of the identify by applying  the b-rule algorithm to a linear system which is associated  with a scalable frame. We also give an estimate of the number of minimal scalings of a scalable frame. We provide 
a characterization of when  minimal scalings are affinely dependent.
 Using this characterization, we can conclude that  all strict scalings $c=(c(1),\dots,c(k))\in \R_{> 0}^k$ of $F$ have the same structural property. That is, the collections of all tight subframes of strictly scaled frames are the same up to a permutation of the frame elements.
We also present the uniqueness of  orthogonal partitioning property of any set of minimal scalings, which provides all possible tight subframes of a given scaled frame. 
\end{abstract}

\maketitle

\section{Introduction}\label{intro}
A frame in $\R^n$ is a spanning set, and a tight frame $\Set{f_i}_{i=1}^k$ with $k \ge n$ is a frame which provides an orthonormal basis-like representation, i.e., 
there exists a positive constant $\lambda$ such that for any $f$ in $\R^n,$
\begin{equation}
\label{recon}
f =  \lambda  \sum_{i=1}^k \inpro{f, f_i} f_i.
\end{equation}
If $\lambda =1$ in (\ref{recon}) then $\Set{f_i}_{i=1}^k$ is said to be a Parseval frame. 
Many early applications of tight frames were  in signal processing. However, nowadays the theory and applications of tight frames have gone beyond pure and applied mathematics to  other areas such as engineering, computer science, and medicine. 
Applications of tight frames are growing because tight frames are redundant systems that have simple reconstruction properties mentioned above and provide optimal numerical stability.  Tight frames  can capture signal characteristics and  are flexible for achieving better approximation and other desirable features. 
One of the active areas of research is  the construction of tight frames.  Various methods of constructing tight frames have been developed for specific types of frames, including unit-norm tight frames, equiangular tight frames, tight frames of vectors having a given sequence of norms,  tight fusion frames, sparse equal norm tight frames using spectral tetris, etc 
\cite{ BM03, STDH07, CMKLT06,  CFHWZ12, HKLW07}.  
 In the last  couple of years the theme of scalable frames have been developed as a method of constructing tight frames from general frames by manipulating the length of frame vectors. 
 Scalable frames  maintain erasure resilience and sparse expansion properties of frames 
 \cite{KOPT13, CC13, KOF13, CKLMNPS14,   CKOPW15}. 
 In this paper,  we further explore scalable frames. 
It is known that the set of all scalings of a frame forms a convex polytope whose vertices correspond to the minimal scalings. 
In this paper, 
we give a method to find a subset of contact points which provides a decomposition of the identity, and an estimate for  the number of minimal scalings of a scalable frame. We provide 
a characterization of when  minimal scalings are affinely dependent.
 Using this characterization, we can conclude that  all strict scalings $c=(c(1),\dots,c(k))\in \R_{> 0}^k$ of $F$ have same tight subframes. 
We also present the uniqueness of orthogonal partitioning property of any set of minimal scalings, which provides all possible tight subframes of a given scaled frame. 
\section{Preliminaries}
In this section we recall  basic properties of tight frames and scalable frames in $\R^n$. We present a few results that will be used later in the
paper. For basic facts about scalable frames we refer to  \cite{KOPT13, CC13, KOF13, CKLMNPS14, CKOPW15, CCNST16}. 

\begin{definition}
A sequence $\{f_i\}_{i=1}^k\subseteq \R^n$, is a frame for $\R^n$ with frame bounds $0<A\leq B<\infty$ if for all $f \in \R^n$,
\begin{equation}\label{framedef}
A|| f ||^2\leq \sum_{i=1}^k |\langle f,f_i\rangle|^2\leq B|| f ||^2 .
\end{equation}
\end{definition}
Throughout this paper, we assume that frame elements are nonzero vectors. 
Often it is useful to express frames both as sequences as well as matrices. Therefore we abuse  notation and denote a frame $F=\{f_i\}_{i=1}^k$  as a $n \times k$ matrix $F$ whose $k$ column vectors are $f_i$, $i = 1, \ldots, k$.

A \emph{unit-norm frame} is a frame such that each vector in the frame has norm one. 
A frame $\{f_i\}_{i\in I}$ is said to be $\lambda-tight$ if $\lambda=A=B$ in \eqref{framedef} and is said to be $Parseval$ if $A=B=1$. 

We note that a frame $F$ is a Parseval frame if and only if 
\begin{equation}\label{Parseval}
FF^t = I_n.
\end{equation}

Let $\{v_i\}_{i \in I}$ be a set of vectors  in $\R^k$. 
 The set of all convex combinations of $\{v_i\}_{i \in I}$ is called the \emph{convex hull} of $\{v_i\}_{i \in I}$ and is defined as
  \begin{align*}
    conv\{v_i\}_{i \in I} :=\left\{\sum_{i \in I}\alpha_iv_i : \alpha_i \geq 0, \sum_{i \in I}\alpha_i=1\right\}.
  \end{align*}
   We also note that    a \emph{polytope} in $\R^k$ is a convex hull of finitely many points in $\R^k$ and 
   the relative interior of $conv\{v_i\}_{i \in I}$, denoted 
   $(conv\{v_i\}_{i \in I})^\circ$, is defined as 
\begin{equation}
\label{eq_1}
 (conv\{v_i\}_{i \in I})^\circ   :=\left\{\sum_{i \in I}\alpha_iv_i : \alpha_i > 0, \sum_{i \in I}\alpha_i=1\right\}.
\end{equation}

A \emph{face} of a convex polytope is any intersection of the polytope with a half space such that none of the relative interior points of the polytope lie on the boundary of the half space. If a polytope is $k$-dimensional, its facets are the $(k -1)$-dimensional faces, its edges are the $1$-dimensional faces, and its vertices are the $0$-dimensional faces. 

 The \emph{affine hull} of $\{v_i\}_{i \in I}$ is defined to be 
$$ \text{aff}\{v_i\}_{i \in I}:=\{\sum_{i \in I}\alpha_iv_i : \sum_{i \in I} \alpha_i=1\}.$$
 The set $\{v_i\}_{i \in I}$ is \emph{affinely dependent} if there exists $i \in I$ such that $v_i \in \mbox{aff}\{v_j\}_{j \in I \setminus\{i\}}$. This is equivalent to the existence of $\alpha_i$, $i \in I$ not all zeros such that 
 both $\sum_{i \in I} \alpha_i v_i =0  $  and $\sum_{i \in I} \alpha_i  =0 $. 

Let $w=(w(1),\ldots,w(k)) \in \R^k$. The \emph{support} of $w$, denoted by $supp(w)$, is defined as $\{i : w(i) \neq 0\}$.

Let $F=\{f_i\}_{i=1}^k$ be a unit-norm  frame in $\R^n$. We call $$c=(c(1),\ldots,c(k)) \in \R^k_{\geq 0}$$ a \emph{scaling} of $F$ 
if the scaled frame $ \{ \sqrt{c(i)} f_i\}_{i=1}^k$ is a Parseval frame for $\R^n$. We denote the scaled frame by $cF$.  If a scaling exists, then the unit-norm frame $F$ is said to be \emph{scalable}. If $c$ is a scaling with $supp(c)=\{1,\ldots,k\}$, then $c$ is called a \emph{strict} scaling and the unit-norm frame $F$ is said to be \emph{strictly scalable}.
A scaling $c$ is a \emph{minimal scaling} if $\{f_i : c(i)>0\}$ has no proper scalable subframe. 
We denote the set of all scalings and the set of all minimal scaling  of a scalable frame $F$ by  $\mathcal{C} (F)$ and  
$\mathcal{M} (F)$, respectively.

For any vector \(f\in\mathbb{R}^n\), we define the diagram vector associated with \(f\), denoted \(\tilde{f}\), by
\begin{equation*}
\tilde{f} := 
\frac{1}{\sqrt{n-1}}
\begin{bmatrix}f^2(1)-f^2(2)\\  \vdots  \\ f^2(n-1)-f^2(n) \\  
\sqrt{2n}f(1)f(2) \\ \vdots  \\  \sqrt{2n}f(n-1)f(n)
 \end{bmatrix}\in\mathbb{R}^{n(n-1)},
\end{equation*}
where the difference of squares 
$f^2(i)- f^2(j)$ and the 
 product \(f(i)f(j)\)  occur exactly once for \(i < j, \ i = 1, 2, \cdots, n-1.\) 
The diagram vectors give us the following characterization of a tight frame.
\begin{theorem}[\cite{CKLMNS13,CKLMNPS14}]
\label{charTight}
Let \(\{f_i\}_{i=1}^k\) be a sequence of vectors in \(\mathbb{R}^n\), not all of which are zero. Then \(\{f_i\}_{i=1}^k\) is a tight frame if and only if \(\sum_{i=1}^k\tilde{f_i}=0\). 
\end{theorem}

We use the diagram vectors of a given unit-norm frame to characterize scalable  frames.

\begin{theorem}[\cite{CKLMNPS14}, Proposition 3.6]\label{scale}
Let  \(\{f_i\}_{i=1}^k\) be  a unit-norm frame for $\R^n$ and  $c=(c(1), \cdots, c(k))$ be a vector in $R^k_{\ge 0}$. 
Let $\tilde{G} := \gp{ \inpro{\tilde{f_j},\tilde{ f_i}} }_{i, j=1}^k$ be the Gramian associated to the diagram vectors $\Set{\tilde{f}_i}_{i=1}^k$. 
Then  $cF $ is a Parseval frame  for $\R^n$ if and only if the vector 
 $c$  
  belongs to the null space of $\tilde{G}$ and $c(1) + \ldots + c(k) =n$. 
\end{theorem}
We note that the condition $c(1) + \ldots + c(k) =n$ in the above theorem is added to Proposition 3.6 in  \cite{CKLMNPS14}  in order for  $cF$ to be a Parseval frame.

\section{Minimal scalings}
A connection between frames and the existence of John's decomposition of the identify have been studied earlier, \cite{CKOPW15, Ver01}. 
In this paper, 
we provide a method to find all possible  contact points for the John's decomposition of the identify by applying  the b-rule algorithm to a linear system which is associated  with a scalable frame from Theorem \ref{scale}. We also give an estimate of the number of minimal scalings of a scalable frame.

Given a scalable frame $F$ the authors of \cite{CC13} showed that the set of all scalings $\mathcal{C}(F) $ is a convex polytope whose vertices correspond to the finite  set of minimal scalings $\mathcal{M}(F) $. 

\begin{theorem}[\cite{CC13}]\label{spolytope}
  Let $F=\{f_i\}_{i=1}^k$ be a unit-norm frame in $\R^n$.  Then we have 
  $$ \mathcal{C}(F) = conv\gp{\mathcal{M}(F) } .$$
\end{theorem}

From (\ref{Parseval}) the polytope 
$$ \mathcal{C}(F) 
 = \Set{  \gp{c(1),\ldots, c(k)} \in \R^k_{\ge 0} \,: \, \sum_{i=1}^kc(i)f_if_i^*=I_n }. $$
 This is called the \emph{scalability polytope} of $F$. 
 
 \begin{theorem}\label{minChar}
 Let $F$ be a scalable frame for $\R^n$ and let $v \in \mathcal{C}(F)$. Then 
 $$  |supp(v)| \le \frac{n(n+1)}{2} \text{ if and only if } v \in \mathcal{M}(F).$$
 \end{theorem}
 
 \begin{proof}
 If $  |supp(v)| \le \frac{n(n+1)}{2} $, then by Corollary 2.2 in \cite{CC13}, 
 $$\Set{f_i f_i^*\,:\, i \in supp(v)}$$ is linearly independent. 
 That is, $\Set{f_i \,:\, i \in supp(v)}$ is scalable
 with the unique scaling $v$ which implies that $v \in \mathcal{M}(F)$.
 
If $v \in \mathcal{M}(F)$, then $\Set{f_i f_i^*\,:\, i \in supp(v)}$ is linearly independent  by Theorem 3.5 in \cite{CC13}. 
Since the dimension of $n\times n$ real symmetric matrices is $\frac{n(n+1)}{2}$, we conclude that 
$  |supp(v)| \le \frac{n(n+1)}{2} $.
 \end{proof}

We now turn our attention to the linear system to find all minimal scalings of a given scalable frame. 
 This linear system provide us a method to find  a subset of the set 
of contact points for John's decomposition of the identity and an estimate for   
 the size  $\ab{\mathcal{M}(F)}$ of minimal scalings.  
 In the following,  we provide an estimation of the number of minimal scalings of a scalable frame using the Gramian associated to the diagram vectors of the frame vectors. 
 Let $F=\{f_i\}_{i=1}^k$ be a unit-norm frame in $\R^n$.  Let $\tilde{G} := \gp{ \inpro{\tilde{f_j},\tilde{ f_i}} }_{i, j=1}^k$ be the Gramian associated to the diagram vectors $\Set{\tilde{f}_i}_{i=1}^k$. 
From Theorem \ref{scale},  we have a second description of  $\mathcal{C}(F)$:
$$ \mathcal{C}(F) 
 = \Set{  \gp{c(1),\ldots, c(k)} \in \R^k_{\ge 0} \,: \, \sum_{i=1}^kc(i)f_if_i^*=I_n } $$
 $$ = \Set{  \gp{x(1),\ldots, x(k)} \in \R^k_{\ge 0} \,:\,  
 \begin{cases}
 \tilde{G} x =0\cr
x(1) + \ldots + x(k) =n
 \end{cases}  } $$

The second characterization of the set of scalings is obtained by solving a linear system, which allows us to adopt a relatively fast algorithm to find the set of minimal scalings \cite{AK04, CC13}.  
 Specifically, by applying the b-rule algorithm (a modification of the simplex  algorithm to find solutions in $\R^k_{\ge 0}$) \cite{AK04} to the linear system
\begin{equation}
 \label{brule}
 \begin{cases}
 \tilde{G} x =0\cr
x(1) + \ldots + x(k) =n
 \end{cases},
\end{equation}
we obtain the set of  minimal scalings  $\mathcal{M}(F)$. 
 \begin{theorem}
  Let $F=\{f_i\}_{i=1}^k$ be a unit-norm frame in $\R^n$ and let $\tilde{G}$ be the Gramian associated to the diagram vectors $\Set{\tilde{f}_i}_{i=1}^k$.  Then we have 
  \begin{equation}\label{Mbound}
   \ab{\mathcal{M}(F)} \le  \gp{ \begin{matrix}
 k \\ rank (\tilde{G}) + 1
 \end{matrix}}.
 \end{equation}
\end{theorem}
\begin{proof}
Note that the system of equations (\ref{brule}) can be reduced to a system of $rank(\tilde{G}) + 1$ equations. When the b-rule algorithm is applied to  
  $\gp{ \begin{matrix}
 k \\ rank (\tilde{G}) + 1
 \end{matrix}}$ systems of equations to find the minimal scalings, it follows that 
$$ \ab{\mathcal{M}(F)} \le  \gp{ \begin{matrix}
 k \\ rank (\tilde{G}) + 1
 \end{matrix}}.$$
\end{proof}

We note that when $F$ is an orthonormal basis,  we obtain  the equality in (\ref{Mbound}).

The following is a well-known characterization of the unique ellipsoid of maximum volume in a convex body in $\R^n$, called the John's ellipsoid theorem.
\begin{theorem}[\cite{GS05}]
\label{john}
Let $E \subset \R^n$ be compact, convex, symmetric in the origin $0$, and with $B^n \subset E$. 
Then the following claims are equivalent:
\begin{enumerate}[(i)]
\item $B^n$ is the unique ellipsoid of maximum volume in $E$.
\item There are $f_i \in B^n \cap bd (E)$ and $c_i >0$, i =1, \ldots, k, where 
$ n \le k \le \frac{1}{2}n (n+1)$, such that 
\begin{equation}
\label{JI}
 I_n = \sum_{i=1}^k c_i f_i \otimes f_i.
\end{equation}
\end{enumerate}
Here, $B^n$ is the solid unit ball in $\R^n$ and $bd(E)$ stands for the boundary of $E$. 
\end{theorem}
We call Equation (\ref{JI})  as the John's decomposition of the identity  and the elements of $B^n \cap bd (E)$ as the contact points. 
The relation between the measure of scalable frames and John's ellipsoid theorem is studied in \cite{CKOPW15}. 
 Some subsets of the set of contact points can be useful in understanding the orthogonal structure under action of a given linear operator \cite{Ver01}. In the following, we study the connection between a minimal scaling of a scalable frame and  subsets of the set of contact points for the John's decomposition of the identity.
The relation of a scalable frame and John's ellipsoid theorem are obtained by rewriting  (\ref{JI}) as the following equation:
$$  f = \sum_{i=1}^k  \inpro{f, \sqrt{c_i}f_i}\sqrt{c_i}f_i, \text{ for any } f \in \R^n. $$
This allows us to consider the subset of contact points in  (\ref{JI}) as a frame in $\R^n$. 
If $F = bd(E) \cap B^n $ is finite, using the system of equations (\ref{brule}) together with Theorem \ref{minChar}, we obtain all subsets of the set of the contact points for the John's decomposition of the identity since the b-rule algorithm finds
all entry-wise nonnegative vectors that are solutions to (\ref{brule}). 
This is stated in the following theorem. 
\begin{theorem}
Let $E \subset \R^n$ be compact, convex, symmetric in the origin $0$. 
Let $F = bd(E) \cap B^n $ be a finite set of contact points. 
 If $B^n$ is the unique ellipsoid of maximum volume in $E$, then 
the  frame vectors corresponding to any minimal scaling of $F$  is a subset of the set of  contact points in  John's decomposition of the identity.
\end{theorem}

As an example, in $\R^2$, let us consider the following  contact points  
$$F =\bk{
 \begin{matrix}
\cos 10  &  -\frac{1}{2} &  -\frac{1}{2}  & -\cos 10  &  \frac{1}{2} &  \frac{1}{2} \cr
\sin 10  &  \frac{\sqrt{3}}{2} &  -\frac{\sqrt{3}}{2} & -\sin 10  & - \frac{\sqrt{3}}{2} &  \frac{\sqrt{3}}{2}  \cr 
  \end{matrix}}. 
  $$

 \begin{figure}
\centering
\begin{tikzpicture}[smooth,scale=2]
\draw[->] (-1.5,0) -- (1.5,0) ; 
\draw[->] (0,-1.5) -- (0,1.5) ; 

\draw (0, 0) circle (1);

\draw [fill] (0.9848,0.173648) circle [radius = .05] node [right]{$f_1$};
\draw[ultra thick, domain=.9:1.11] 
        plot(\x, {-0.9848/0.173648*\x + 0.173648 +0.9848*0.9848/0.173648 }) ;

\draw [fill] (-0.9848, -0.173648) circle [radius = .05] node [above right]{$f_4$};
\draw[ultra thick, domain=-1.11:-0.9] 
        plot(\x, {-0.9848/0.173648*\x - 0.173648 -0.9848*0.9848/0.173648 } );

\draw [fill] (0.5, 0.866) circle [radius = .05] node [above right]{$f_6$};
\draw[ultra thick, domain=0:.9] 
        plot(\x,  {-0.5/0.866*\x +0.866 + 0.5*0.5/0.866 } );

\draw [fill] (-0.5, -0.866) circle [radius = .05]node [above right]{$f_3$};
\draw[ultra thick, domain=-.9:0] 
        plot(\x, {-0.5/0.866*\x -0.866  -0.5*0.5/0.866 } );

\draw [fill] (-0.5, 0.866) circle [radius = .05]node [above left]{$f_2$};
\draw[ultra thick, domain=-1.1: 0] 
        plot(\x, {0.5/0.866*\x +0.866  +0.5*0.5/0.866 }) ;

\draw [fill] (0.5, -0.866) circle [radius = .05]node [below right]{$f_5$};
\draw[ultra thick, domain=0:1.1] 
        plot(\x, {0.5/0.866*\x -0.866  -0.5*0.5/0.866 });

\end{tikzpicture} 
\caption{Convex body with a set of  contact points}

 \end{figure}
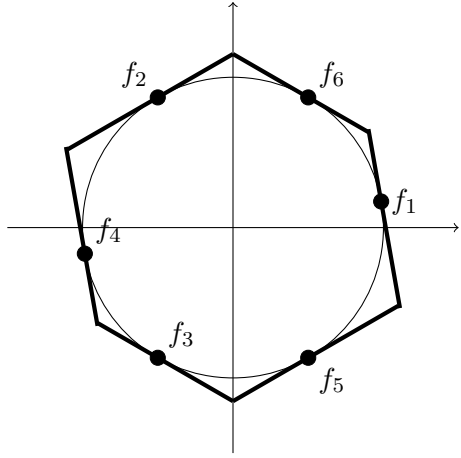
  Then 
 $$ f_1 =  \bk{
 \begin{matrix}
\cos 10  \cr   \sin 10  \cr 
  \end{matrix}}, \quad 
  f_2 =  \bk{
 \begin{matrix}
 -\frac{1}{2}  \cr   \frac{\sqrt{3}}{2}  \cr 
  \end{matrix}}, \quad 
  f_3 =  \bk{
 \begin{matrix}
-\frac{1}{2}  \cr   -\frac{\sqrt{3}}{2}  \cr 
  \end{matrix}}, 
     $$
  together with 
    $$
  c_1=\frac{2}{3 \cos^210 - \sin^2 10}, $$
  $$
  c_2 = \frac{2 \sqrt{3}}{3} \frac{\sqrt{3}\cos^210 - \sqrt{3}\sin^210 + 2 \cos10 \sin10}{3 \cos^210 - \sin^2 10}, $$
 $$ c_3 = \frac{2 \sqrt{3}}{3} \frac{\sqrt{3}\cos^210 - \sqrt{3}\sin^210  - 2 \cos10 \sin10}{3 \cos^210 - \sin^2 10}
 \vspace{.1in}
   $$
satisfy the second statement of Theorem \ref{john}. 
We note that $ \ab{\mathcal{M}(F)} =16$ and all of the minimal scalings satisfy  John's decomposition of the identity.

\section{Structural properties of scalable frames}
In  subsection \S4.1,  we study some properties of general polytopes, which provide a characterization of affine  dependency of  minimal scalings in subsection \S 4.2. 
We show that if minimal scalings are affinely independent, all strict scalings of a frame  have the same structural property.  
That is, the collections of all tight subframes of strictly scaled frames are the same up to a permutation of the frame elements.

\subsection{General polytopes}

\begin{proposition}
\label{5_10}
  Let $\{v_i\}_{i \in I}$ be the set of vertices for a polytope. Then  $\{v_i\}_{i \in I}$ is affinely dependent if and only if 
  $$  (conv\{v_j\}_{j \in J_1})^\circ \cap (conv\{v_j\}_{j \in J_2})^\circ \neq \emptyset $$
   for some disjoint subsets $J_1,J_2 \subseteq I.$
\end{proposition}
\begin{proof}
  ($\Leftarrow$)
  Let $\sum_{j \in J_1}\alpha_jv_j=\sum_{j \in J_2}\alpha_jv_j$,  where  $\alpha_j>0$,  
  $ \sum_{j \in J_1}\alpha_j = \sum_{j \in J_2}\alpha_j = 1$. 
  Then  $\sum_{j \in J_1}\alpha_jv_j -\sum_{j \in J_2}\alpha_jv_j =0$ and 
  $ \sum_{j \in J_1}\alpha_j - \sum_{j \in J_2}\alpha_j = 0$.  
  We conclude  $\{v_i\}_{i \in J_1 \cup J_2}$ is affinely dependent and hence $\{v_i\}_{i \in I}$ is affinely dependent. 

  ($\Rightarrow$) Since $\{v_i\}_{i \in I}$ is affinely dependent, there exists $i \in I$ such that $v_i \in \mbox{aff}\{v_j\}_{j \in I \setminus\{i\}}$. 
We write $v_i= \sum_{j \in J_1}\alpha_jv_j + \sum_{j \in J_2}\alpha_jv_j$, where   $\alpha_j$ is positive for $j \in J_1$, negative for $j \in J_2$, and $J_1 \cupdot J_2 \subseteq I \setminus \Set{i}$.  
Since $ \sum_{j \in J_1 \bigcupdot J_2 } \alpha_j =1 $, $J_1 \neq \emptyset$.  
Since $v_i$ is a vertex of the polytope, $J_2 \neq \emptyset$.   
Let $ r =   \sum_{j \in J_1  } \alpha_j$, then 
$$  w = \frac{1}{r}  \sum_{j \in J_1  } \alpha_j  v_j \in (conv\{v_j\}_{j \in J_1})^\circ  \text{ and} $$
$$ w = \frac{1}{r}  \left( v_i  + \sum_{j \in J_2  } (-\alpha_j)  v_j \right)  \in (conv\{v_j\}_{j \in J_2 \cup \{i\}})^\circ. $$
This completes the proof. 
\end{proof}

\begin{proposition}
\label{5_11}
  Let $\{v_i\}_{i \in I}$ be the set of vertices for a polytope and let  $conv\{v_j\}_{j \in J}$ be a nontrivial face. 
 If $\sum_{i \in I} \alpha_iv_i \in conv\{v_j\}_{j \in J}$, then 
 $\alpha_i=0$ for $i \in I \setminus J$. 
 
\end{proposition}
\begin{proof}
Let $\mathcal{H}=\{x \in \R^k : \inpro{c, x} = b, \  b \in \R, \ c \in  \R^k\setminus \{0\} \}$ be the supporting hyperplane containing $conv\{v_j\}_{j \in J}$. 
We write $\sum_{i \in I} \alpha_iv_i = \sum_{i \in J} \alpha_iv_i + \sum_{i \in I\backslash J} \alpha_iv_i$.  
Suppose that $\alpha_{i_0} \neq 0$ for some $ i_0 \in I \backslash J$. 
Since $ \sum_{i \in J} \alpha_iv_i  \in \mathcal{H}$ and $ \sum_{i \in I \backslash J} \alpha_iv_i  \notin \mathcal{H}$  , 
$$ \inpro{c,  \left(   \sum_{i \in J} \alpha_iv_i \right) }= b \text{ and }
\inpro{ c ,   \left(   \sum_{i \in I \backslash J} \alpha_iv_i \right) } <b. $$
This implies that $ \sum_{i \in I} \alpha_iv_i  \notin \mathcal{H}$. This completes the proof.
\end{proof}

\begin{corollary}
\label{5_12} 
 Let $\{v_i\}_{i \in I}$ be the set of vertices for a polytope. 
Let $J_1,J_2$ form a partition of $I$ such that 
 $$ (conv\{v_j\}_{j \in J_1})^\circ \cap (conv\{v_j\}_{j \in J_2})^\circ \neq \emptyset.$$ 
 Then $conv\{v_j\}_{j \in J_1}$ and $conv\{v_j\}_{j \in J_2}$ are not faces of the polytope.
\end{corollary}

If we have non negativity in each entry of the vertices of a polytope in $\R^k$, 
we obtain the affine dependency of vertices from the relation of supports of the vertices.

\subsection{Properties of the minimal scalings}
In this section, we provide a characterization of when the minimal scalings are affinely dependent. Using this characterization, we can conclude that  all strict scalings of a given frame have the same 
 tight subframes up to a permutation. 
We also present the uniqueness of orthogonal partitioning property of any set of minimal scalings, which provides all possible tight subframes of a given scaled frame. 

\begin{lemma}
\label{5_15}
Let $\{v_i\}_{i \in I}$ be the set of minimal scalings of a scalable frame $F$.
 Suppose $w=\sum_{i \in I}\alpha_iv_i$ is an affine combination of $\{v_i\}_{i \in I}$ and $w$ has all nonnegative entries. Then $w \in \mathcal{C}(F)$.
\end{lemma}
\begin{proof}
Let $F = \Set{f_i}_{i=1}^k$. 
 Recall that 
 $$\mathcal{C}(F)=\{c \in \R^k : c(i) \geq 0, \, \sum_{i=1}^kc(i)f_if_i^*=I_n\}.$$
 Since $v_j = ( v_j(1), \ldots, v_j(k))$ is  a minimal scaling of $F$, 
  we have 
  $$\sum_{i=1}^kv_j(i)f_if_i^*=I_n.$$
 Thus $\sum_{i=1}^kw(i)f_if_i^*=\sum_{i=1}^k(\sum_{j \in I}\alpha_jv_j)(i)f_if_i^*=\sum_{j \in I}\alpha_j\sum_{i=1}^kv_j(i)f_if_i^*=\sum_{j \in I}\alpha_j I_n=I_n$ since $\sum_{j \in I}\alpha_j =1$. Therefore, $w \in \mathcal{C}(F)$. 
\end{proof}

\begin{proposition}
\label{5_16}
 Let $\{v_i\}_{i \in I}$ be the set of minimal scalings of a scalable frame. 
If  $supp(v_{i_0}) \subseteq \cup_{j \in I \setminus\{i_0\}} supp(v_j)$ for some $i_0 \in I$, 
then $ v_{i_0} \in \text{aff} \{v_i\}_{i \in I\backslash \{i_0\}}$. 
\end{proposition}
\begin{proof}
 Let $J=I \setminus \{i_0\}$ and $w = \frac{1}{|J|} \sum_{j\in J} v_j$.  Set 
 $$ \epsilon = \frac{\text{min} \{ w(l) \,:\, w(l) >0\} }{ \text{max} \{ v_{i_0}(l) \,:\,  v_{i_0}(l) >0\}  }. $$
If $w(m) =0$ for some $m = 1, \cdots, k$, then $v_{i_0} (m) =0$ since $supp(v_{i_0}) \subseteq \cup_{j \in I \setminus\{i_0\}} supp(v_j)$. This implies that $((1+\epsilon) w - \epsilon v_{i_0}) (m) =0$. It is clear that if $v_{i_0} (m) =0$, then  $((1+\epsilon) w - \epsilon v_{i_0}) (m) \ge 0$. 
For each $m = 1, \cdots, k,$ such that $v_{i_0} (m) \neq 0$, we also have 
$$ ((1+\epsilon) w - \epsilon v_{i_0}) (m) \ge
(1+\epsilon) w(m) -   \frac{\text{min} \{ w(l) \,:\, w(l) >0\} }{ v_{i_0}(m)   } v_{i_0} (m) \ge 0.  $$
Since $\theta = (1+\epsilon) w - \epsilon v_{i_0} $ is an affine combination of minimal scalings and 
$\theta (m) \ge 0$ for $m=1, \ldots k$, we conclude from Lemma \ref{5_15} that $\theta \in \mathcal{C}(F)$. 
Thus, we have 
  $ \theta = \sum_{i \in I} \alpha_i v_i$ with $\alpha_i \ge 0$ and $\sum_{i \in I} \alpha_i =1$. 
It follow from the expansion that  
$$v_{i_0} =  \sum_{j \in J} \gp{ \frac{(1+\epsilon)/|J| -\alpha_j }{ \epsilon + \alpha_{i_0}} } v_j
\text{ and }
\sum_{j \in J} \frac{(1+\epsilon)/|J| -\alpha_j }{ \epsilon + \alpha_{i_0}} =1,$$
which completes the proof. 
\end{proof}

 \begin{remark}
 We remark  that if  $\{v_i\}_{i \in I}$ is not the set of minimal scalings, then in general Proposition \ref{5_16} is not true. For example, let 
 $$ v_1=(1,0), \, v_2 =(0, 1), v_3=(1,1) $$
 be the vertices of a polytope. 
 Then $supp(v_1) =\Set{1} \subset \Set{1, 2} =supp(v_2) \cup supp(v_3)$, but 
 $v_1 \notin \text{aff} \Set{v_2, v_3}$.
 \end{remark}

 Since the  minimal scalings of a scalable frame is  the set of vertices of a polytope and each entry of the vertices is  non negative, from the propositions in section \S 4.1 and Proposition \ref{5_16}, we have the following equivalent formulations of affine dependency of minimal scalings: 
 
\begin{theorem}
\label{5_17}
 Let $\{v_i\}_{i \in I}$ be the set of minimal scalings of a scalable frame. 
 Then the following are equivalent:
  \begin{enumerate}[1.]
    \item The set of minimal scalings $\{v_i\}_{i \in I}$ is affinely dependent.
    \item There exists $i \in I$ such that $supp(v_i) \subseteq \cup_{j \in I \setminus \{i\}}supp(v_j)$.
     \item There exist disjoint $J_1, J_2 \subseteq I$ such that 
    $$(conv\{v_j\}_{j \in J_1})^\circ \cap  (conv\{v_j\}_{j \in J_2})^\circ 
    \neq 
    \emptyset.$$
    \item There exist disjoint $J_1, J_2 \subseteq I$ such that 
    $$\cup_{j \in J_1}supp(v_j)=\cup_{j \in J_2}supp(v_j).$$ 
  \end{enumerate}
\end{theorem}
\begin{proof}
The relation $2 \Rightarrow 1 \Rightarrow 3$ follows from Proposition \ref{5_16} and Proposition \ref{5_10}. \\

\noindent
$3 \Rightarrow 4$.  Let 
$ w \in (conv\{v_j\}_{j \in J_1})^\circ \cap  (conv\{v_j\}_{j \in J_2})^\circ $, then we have 
$$ supp(w) =  \cup_{j \in J_1}supp(v_j)=\cup_{j \in J_2}supp(v_j).$$

\noindent
$4 \Rightarrow 2$.  Let $i \in J_1$, then we have 
$$ supp(v_{i}) \subset   \cup_{j \in J_1}supp(v_j) = \cup_{j \in J_2}supp(v_j) \subset 
 \cup_{j \in I \setminus  \Set{i}}supp(v_j) .$$
\end{proof}

  In the following, we present a series of relations between  minimal scalings and structural properties of a scaled frame.
In order to state these results,   we need the notion of an empty cover of the factor poset of a frame found in 
\cite{BCEKKNN16, CCNST16}. 
The factor poset corresponds to tight subframes of $F$ and the empty cover corresponds to the minimal tight subframes of $F$. 

\begin{definition}
Let  $F= \Set{f_i}_{i \in I}$ be a finite frame in $\R^n$. 
We define its factor poset ${\mathbb F}(F) \subset 2^I$ to be  the set 
$${\mathbb F}(F) := \Set{ J \subset I \,:\, \ \Set{f_j}_{j \in J} \text{ is a tight frame for } \R^n} 
\cup \Set{\emptyset} $$
partially ordered by inclusion. We define the empty cover of ${\mathbb F}(F)$, $EC(F)$, to be the set of $J \in {\mathbb F}(F)$ which covers 
$\emptyset$, that is, 
$$ EC(F) := \left\{ J \in {\mathbb F}(F)\,:\, J \neq \emptyset \text{ and }  \nexists J' \in {\mathbb F}(F)  \right.$$
$$
\left. \text{ with } \emptyset \subsetneq J' \subsetneq J \right\}. $$
\end{definition}

For example, consider the following frame in $\R^2$, 
 $$F =\bk{
 \begin{matrix}
 1 & 0 & -1 & 0 \cr
 0 & 1 & 0 & -1
  \end{matrix}}.
  $$
Then 
$${\mathbb F}(F) = \Set{ \emptyset, \Set{1, 2}, \Set{1, 4},\Set{2, 3},\Set{3, 4},\Set{1, 2, 3, 4} }, $$
$$ EC(F) = \Set{  \Set{1, 2}, \Set{1, 4},\Set{2, 3},\Set{3, 4}}. $$

The following theorem shows that ${\mathbb F}(F) $ can be obtained by taking disjoint union of subsets of $EC(F)$. 

\begin{theorem} [\cite{BCEKKNN16}]
\label{2_6}
If $F$ is a frame, then  
\[\mathbb{F}(F)=\Set{\bigcupdot_{E \in S} E  : S\subseteq EC(F) }
.\]
\end{theorem}

  A scaling of a unit-norm frame $F$ is \emph{prime} if the scaled frame $cF$ does not contain any proper, tight subframes and \emph{non-prime} otherwise. The following theorem was proved in \cite{CCNST16}.
\begin{theorem}[\cite{CCNST16}]
 \label{5_19}
 A scaling is non-prime if and only if it is a convex combination of minimal scalings which can be partitioned into two orthogonal subsets. 
 \end{theorem}

Motivated by Theorem \ref{5_19}, we study for a scalable frame $F$ the connection between orthogonal partitioning of minimal scalings and the tight subframes of the scaled frame $cF$.  
We define the smallest orthogonal partition of minimal scalings  $\{v_i\}_{i \in I}$ to be a partition  
 $$ \Set{ \{v_j\}_{j \in J_1}, \ldots, \{v_j\}_{j \in J_a} } $$
 such that
 $J_1 \cup \ldots \sup J_a = J \subseteq I$ and 
  the subsets in the collection are mutually orthogonal 
 (i.e., 
 $ \inpro{ v_i , v_j} =0 \text{ if } i \in J_k, j \in J_l, \text{ and } l \neq k $). 
 Moreover each subset cannot be  partitioned further into orthogonal subsets. 

Suppose $\{v_j\}_{j \in J}$ can be written as 
 \begin{equation}
 \label{twoDecom}
 \{v_j\}_{j \in J}=\{v_j\}_{j \in J_1}\cup \ldots \cup \{v_j\}_{j \in J_a}
  \end{equation}
   \begin{equation}
    \label{twoDecom2}
  =\{v_j\}_{j \in K_1} \cup \ldots \cup \{v_j\}_{j \in K_b},
 \end{equation}
where each collection is a  smallest orthogonal partition of $\{v_j\}_{j \in J}$ for some $J \subset I$. 
If  $J_1 \neq K_1$, then without loss of generality assume that $J_1 \setminus K_1  \neq \emptyset$. 
Then we have 
$$ J_1 =   \gp{J_1 \setminus K_1} \cup \gp{J_1 \cap K_1}.$$
This is a contradiction to the assumption that $J_1$ cannot be partitioned into orthogonal subsets. 
Thus $J_1 = K_1$.  This shows that the supports of the partition in (\ref{twoDecom}) and  (\ref{twoDecom2}) are the same. Hence $a=b$. Therefore we can now state the following theorem (which also appears in 
\cite{DKN15}).

  \begin{theorem}
 \label{5_22}
  Let $\{v_i\}_{i \in I}$ be the set of minimal scalings of a scalable frame. 
  The smallest orthogonal partition of any subset of $\{v_i\}_{i \in I}$ is unique. 
 \end{theorem}

\begin{observation}
Let $F$ be a scalable frame and $\{v_i\}_{i \in I}$ be the set of minimal scalings. 
Suppose $E \in EC(F)$. 
Define $c\in \R^k$ by 
$c(i)=
\begin{cases}
1 & \text{ if }i \in E,\cr
0 & \text{ otherwise}.
\end{cases}
$
Then $c\in {\mathcal C}(F)$ and $c =\sum_{j \in J} \alpha_j v_j$ for some $J\subset I$. 
From this it follows that 
$E = \cup_{j \in J} supp(v_j)$.
\end{observation}

We now state the theorem about unique orthogonal partitioning property. 
Statement of Theorem \ref{orthoDecom} appears in \cite{DKN15}. Its proof  is presented only in this paper.

\begin{theorem}
\label{orthoDecom}
Let $\{v_i\}_{i \in I}$ be the set of minimal scalings of a scalable frame $F$ 
and let $c $ be a scaling of $F$. 
Suppose $c = \sum_{j \in J} \alpha_j v_j$ such that 
$J \subseteq I$ and $\alpha_j >0$ with $\sum_{j \in J} \alpha_j =1$.
Then $\{v_i\}_{ij\in J}$ can be orthogonally partitioned as 
\begin{equation}
\label{decom}
 c = \sum_{i \in J_1} \alpha_i v_i + \ldots + \sum_{i \in J_a} \alpha_i v_i,  
 \end{equation}
where $ \cup_{i \in J_j} supp(v_i) $ for $ j = 1, \ldots, a $ are pairwise disjoint subsets of $EC(cF)$.  
If $EC(cF)$ is pairwise disjoint, then 
$\{v_j\}_{j \in J_1}\cupdot \ldots \cupdot \{v_j\}_{j \in J_a}$ is the  smallest orthogonal partition of
 $\{v_i\}_{i \in J_1 \cupdot \ldots \cupdot J_a}$ so that the orthogonal decomposition in (\ref{decom}) is unique. 
\end{theorem} 

\begin{proof}
Since $cF$ is a Parseval frame  $supp(c) \in \mathbb{F}(cF)$. From Theorem \ref{2_6}, 
$$ supp(c) = E_1 \cupdot \ldots \cupdot E_a, \quad E_i \in EC(cF). $$
Note that the subframe $\Set{ \sqrt{c(i)} f_i}_{i \in E_j}, \ j=1, \ldots a$, is only a tight subframe but not Parseval in general. However, there exists 
$\lambda_j >0 $ such that $\Set{ \sqrt{ \lambda_j c(i)} f_i}_{i \in E_j}, \ j=1, \ldots a$  is Parseval.  
For each $j =1, \ldots, a$, set  $c_j \in \R^k_{\ge 0}$ by
$$ c_j (i) := \begin{cases} 
{\lambda_j c(i)}  & \text{ if } i \in E_j  \cr
0 & \text{otherwise}.
\end{cases}
$$
Then since  $c_j$ is a scaling of $F$, $c_j = \sum_{i \in J_j} \alpha_i v_i$ for some $\alpha_i >0$ and $J_j \subset I$.
This implies that $c$ can be orthogonally partitioned as follows: 
$$ c =\sum_{j=1}^a \frac{c_j}{\lambda_j} =   \frac{1}{\lambda_j} \sum_{j=1}^a 
\gp{ \sum_{i \in J_j} \alpha_i v_i  },$$
where $\cup_{i \in J_j} supp(v_i) = E_j$.
We now suppose that $EC(cF)$ is pairwise disjoint. 
 Let $\{v_j\}_{j \in K_1}\cupdot \ldots \cupdot \{v_j\}_{j \in K_b}$ be  the  smallest orthogonal partition of
 $\{v_i\}_{i \in J_1 \cupdot \ldots \cupdot J_a}$. 
 To show that $\{v_j\}_{j \in K_1}\cupdot \ldots \cupdot \{v_j\}_{j \in K_b}$ and $\{v_j\}_{j \in J_1}\cupdot \ldots \cupdot \{v_j\}_{j \in J_a}$ are the same orthogonal partition of
 $\{v_i\}_{i \in J_1 \cupdot \ldots \cupdot J_a}$, we redorder $K_1, \ldots, K_b$ and $J_1, \ldots, J_a$ such that for $i <j$
 $$ \min \Set{s \,:\, v_s \in K_i} <   \min \Set{s \,:\, v_s \in K_j} \text{ and } $$
 $$ \min \Set{s \,:\, v_s \in J_i} <   \min \Set{s \,:\, v_s \in J_j}.$$
Note that  $v_1 \in  \{v_j\}_{j \in J_1} \cap \{v_j\}_{j \in K_1} $. 
 We now show that 
 $ \{v_j\}_{j \in J_1} = \{v_j\}_{j \in K_1}.$ 
 Suppose that $  \{v_j\}_{j \in K_1} \subsetneq  \{v_j\}_{j \in J_1}$. Then 
 $$c_1 = \sum_{i \in J_1 \setminus K_1} \alpha_i v_i + \sum_{i \in K_1} \alpha_i v_i.$$
 Since $\cup_{i \in J_1} supp(v_i) = E_1 \in EC(cF)$ the above equation if true produces non empty subsets of $E_1$ in ${\mathbb F}(cF)$,  which is a contradiction. 
 Similarly, $ \{v_j\}_{j \in J_i} \cap \{v_j\}_{j \in K_i} \neq \emptyset$ implies that  
 $ \{v_j\}_{j \in J_i} = \{v_j\}_{j \in K_i}$ for $i=1, \ldots, a$. 
 This shows   $a =b$ and the uniqueness of the decomposition. 
\end{proof}
From theorem \ref{orthoDecom} 
we note that if ${\mathcal M}(F)$ is the set of minimal scalings of a scalable frame $F$, 
then for any 
$c\in \mathcal{C}(F)$, we can obtain all tight subframes of $cF$ using Theorem \ref{2_6}.
 Theorem \ref{5_19}  also tells us the conditions for $c$ under which the set $EC(cF)$ is $\{\emptyset,\{1,\ldots,k\}\}$. 
 Moreover, Theorem \ref{orthoDecom} gives conditions for $c$ under which the empty cover of $cF$ is pairwise disjoint. That is, if we  have two different collection of subsets  of minimal scalings for the orthogonal decomposition (\ref{decom}), then $EC(cF)$ is not pairwise disjoint.  
 We note the orthogonal decomposition (\ref{decom})  is not unique in general.  
 For example, consider the following frame in $\R^2$, 
 $$F =\bk{
 \begin{matrix}
 1 & 0 & 1 & 0 & 0 & 1 \cr
 0 & 1 & 0 & 1 & 1 & 0
  \end{matrix}}.
  $$
The minimal scalings are 
 
 $$
 \begin{matrix}
 v_1 = ( 1, 1, 0, 0, 0, 0), & \quad v_6 = ( 0, 1, 0, 0, 0, 1),\cr
v_2 = ( 0, 0, 1, 1,  0, 0),& \quad v_7 = ( 0, 1, 1, 0, 0, 0),\cr
v_3 = (  0, 0, 0, 0, 1, 1), & \quad v_8= ( 1, 0, 0, 0, 1, 0),\cr
 v_4 = ( 0, 0, 0, 1, 0, 1), & \quad v_9 = ( 1, 0, 0, 1, 0, 0).\cr
 v_5 = ( 0, 0, 1, 0, 1, 0), & \quad \cr
   \end{matrix}
 $$
 Then  the scaling $c =\frac{1}{3}\gp{ 1, 1, 1, 1, 1, 1} $ has the following distinct  orthogonal decompositions:
 $$ c = \gp{\frac{1}{3} v_1} + \gp{\frac{1}{3} v_2} + \gp{\frac{1}{3} v_3}  $$
 $$ =\gp{\frac{1}{3} v_1} + \gp{\frac{1}{3} v_4} + \gp{\frac{1}{3} v_5}.  $$
 The two different orthogonal decompositions of a scaling $c$  guarantees that $\{v_i\}_{i \in J_1 \cupdot \ldots \cupdot J_a}$ is affinely dependent.

\begin{theorem}
\label{5_20}
Let $\{v_i\}_{i \in I}$ be the set of minimal scalings of a scalable frame $F$ and $c$ be a scaling. 
If $\{v_i\}_{i \in I}$  is affinely independent, then  $EC(cF)$ is pairwise disjoint.
\end{theorem}
\begin{proof}
Suppose that $EC(cF)$ is not pairwise disjoint. Then there are two different sets $E_1, \, E_2 \in EC(cF)$ such that $E_1 \cap E_2 \neq \emptyset$. 
Let $$\cup_{j \in J_1} supp(v_j) = E_1, \quad \cup_{j \in J_2} supp(v_j) = E_2. $$
Since $E_1 \neq E_2$, without loss of generality, we assume that $1\in J_1$  and $supp(v_1) \nsubseteq E_2$ so that $1\notin J_2$. 
By Theorem \ref{charTight}, $supp(c) \setminus E_2 \in \mathbb{F}(cF)$ so that 
$supp(c) \setminus E_2 = \cup_{j \in S} supp(v_j) $ for some $S \subset I$. 
Thus $supp(v_1) \nsubseteq \cup_{j \in S} supp(v_j) $ so that $1 \notin S$.  
But $\gp{supp(c) \setminus E_2} \cup  \bigcup_{j \in J_2} supp(v_j)  =  supp (c)$ which implies that 
$$ supp(v_1) \subseteq \bigcup_{j \in S} supp(v_j)  \cup \bigcup_{j \in J_2} supp(v_j) \subseteq \cup_{j \in I\setminus \Set{1}} supp(v_j).$$
Then by Theorem \ref{5_17}, $\{v_i\}_{i \in I}$  is affinely dependent.  
\end{proof}
  
Recall that $c$ is a strict scaling $c$  if $supp(c)=\{1,\ldots,k\}$. 
It is not necessary for a strict scaling $c$ to be a convex combination with contribution from all of the minimal scalings  $\{v_i\}_{i \in I}$. 
However, if $\{1,\ldots,k\}$ is the union of  the support of all minimal scalings, 
a strict scaling $c$ must have all positive coefficients in the convex combination of minimal scalings. 

\begin{proposition}
\label{5_23}
Let $\{v_i\}_{i \in I}$ be the set of minimal scalings of a scalable frame $F$ and $c$ be a strict scaling. 
Suppose that $\{v_i\}_{i \in I}$ are affinely independent. 
Then all the coefficient of the convex combination of minimal scalings for $c$ are positive. 
\end{proposition}
\begin{proof}
Let $c = \sum_{j \in I} \alpha_j v_j$.
Suppose $\alpha_i= 0$ for some $i \in I$. 
Then $c = \sum_{j\in I\backslash \Set{i}} \alpha_j v_j$. 
Since $c$ is a strict scaling, 
$$supp(v_i) \subset supp(c) = \cup_{j\in I \setminus \{i\}}supp(v_j), $$  which contradicts the assumption.

\end{proof}

We remark that when the sets in $EC(F)$ are pairwise disjoint then $\{v_i\}_{i \in I}$ are affinely independent. 
If $\{v_i\}_{i \in I}$ are affinely independent, then all strict scalings give the same poset structure of the scaled frames.

\begin{theorem}
\label{5_24}
Let $\{v_i\}_{i \in I}$ be the set of minimal scalings of a scalable frame $F$ 
which are affinely independent. 
Then for any strict scalings $c_1$ and $c_2$,  we have 
$$ EC(c_1F) =EC(c_2F).$$
Furthermore, $ EC(c_1F)$ is pairwise disjoint. 
\end{theorem}

\begin{proof} By Theorem \ref{5_20} and Theorem \ref{orthoDecom},  both $ EC(c_1F) $ and $EC(c_2F) $ are pairwise disjoint and 
the orthogonal decompositions 
$$c_1 = \sum_{j \in J_1} \alpha_j v_j + \ldots + \sum_{j \in J_a} \alpha_j v_j,$$
$$c_2 = \sum_{j \in K_1} \alpha_j v_j + \ldots + \sum_{j \in K_b} \alpha_j v_j$$
provide the smallest orthogonal partitions of
 $\{v_i\}_{i \in I}$. 
Since $$EC(c_1F) = \Set{\cup_{j \in J_i} supp(v_j) \,:\,  i=1, \ldots, a}, $$ 
 $$EC(c_2F) = \Set{\cup_{j \in K_i} supp(v_j) \,:\,  i=1, \ldots, b}, $$ and the smallest orthogonal partitions of
 $\{v_i\}_{i \in I}$ is unique,  we have 
$$ EC(c_1F) =EC(c_2F).$$
\end{proof}

The following conjecture asserts the existence of a ``maximal'' strict scaling whose factor poset contains all possible factor posets of any strict scaling. 
A maximal strict scaling might be useful to construct a frame in signal processing when we need more representations in certain directions, for example in edge detection or noise detection in image processing. 
\begin{conj}
Let $J \subset I$ such that $\cup_{j \in J}supp(v_j)=\{1,\ldots,k\}$ and $\nexists J_0\subsetneq J$ with $\cup_{j \in J_0}supp(v_j)=\{1,\ldots,k\}$. Let $\{v_j\}_{j \in J_1}\cupdot \ldots \cupdot \{v_j\}_{j \in J_a}$ be the  smallest orthogonal partition of $\{v_j\}_{j \in J}$. 
 Then  there exists a scaling $c$ such that 
$$EC(cF) = \Set{ \cup_{j \in J_i} supp(v_j) \,:\, i = 1, \ldots, a }.$$
\end{conj}
This conjecture is equivalent to determining whether or not the following is true: if $\bigcup_{j \in J_1}supp(v_j)=\ldots=\bigcup_{j \in J_\ell}supp(v_j)=\{1,\ldots,k\}$, and for each $J_i$, there does not exist $J_0 \subsetneq J_i$ such that $\bigcup_{j \in J_0}supp(v_j)=\bigcup_{j \in J_i}supp(v_j)$, then $(conv\{v_j\}_{j \in J_1})^\circ \cap \ldots \cap (conv\{v_j\}_{j \in J_\ell})^\circ \neq \emptyset$. Based on results in polytope theory (Helly's Theorem, \cite{Ma02}), the assumptions seem too weak for the result to be true. However,  a counterexample or  a weaker result would be a substantial progress.

We end this section with the following observations related to the construction of scalable frames. 
As a consequence we would like to point out that if a vector gets repeated in a scalable frame \(\{f_i\}_{i=1}^k\), then the size of the minimal scalings $ \ab{\mathcal{M}(F)} $ doubles. 

\begin{observation}
\label{5_1}
Let $ \{f_i\}_{i \in K}$ be a unit-norm frame and $K_0 \subset K$.
If $\{f_i\}_{i \in K\setminus K_0}$ is scalable, then 
$$ \mathcal{C}\gp{ \{f_i\}_{i \in K\setminus K_0} } 
= \Set{ c |_{ K\setminus K_0} \,:\, c \in \mathcal{C}\gp{\{f_i\}_{i \in K}}, c(i) =0, i \in K_0 }. $$
\end{observation}

\begin{observation}
Let $\mathcal{M}(F) $ be the set of minimal scalings of of a scalable frame $ F = \{f_i\}_{i \in K}$ and let $K_0 \subset K$.
If $\{f_i\}_{i \in K\setminus K_0}$ is scalable, then the minimal scalings of $\{f_i\}_{i \in K\setminus K_0}$  is the set 
$$
 \Set{ v |_{ K\setminus K_0} \,:\, v \in \mathcal{M}(F), v(i) =0, i \in K_0 }. $$
\end{observation}

\begin{observation}
\label{5_2}
Let $\mathcal{M}(F)$ be the set of minimal scalings of a scalable frame $ F=\{f_i\}_{i=1}^k$ and
let $f_{k+1} = f_i$ for some $i =1, \ldots, k$. 
Then the minimal scalings of \(\{f_i\}_{i=1}^{k+1}\) is the set \\
$  \Set{ \begin{array}{ll} (v(1), \ldots, v(k), 0)   \text{ or }  \\  (v(1), \ldots, v(i-1), 0, v(i+1), \ldots, v(k), v(i))  \end{array}
  \,:\, v \in  \mathcal{M}(F)}.$
\end{observation}

\section*{Acknowledgement}

Much of this work was done when 
R. Domagalski, H. Suh, and X. Zhang participated in the Central Michigan University NSF-REU program in the summer of 2014. A. Chan was the graduate student mentor. Kim was supported by the Central Michigan University FRCE Research Type A Grant \#C48143.

\end{document}